\documentclass[11pt]{article}
\usepackage{graphicx}
\usepackage{fullpage}
\usepackage{amsmath}
\usepackage{amssymb}
\usepackage{amsthm}
\usepackage{epsfig}
\usepackage{color}
\definecolor{clemson-orange}{RGB}{234,106,32}
\definecolor{broncos-orange}{RGB}{252,76,2}
\definecolor{cincinnati-red}{RGB}{190,0,0}
\definecolor{pink}{RGB}{255,105,180}
\definecolor{celtics}{RGB}{46,123,59}
\definecolor{leafs-blue}{RGB}{0,58,120}
\definecolor{pure-cyan}{RGB}{0,100,92}
\definecolor{clemson-orange}{RGB}{234,106,32}
\definecolor{chicago-maroon}{RGB}{128,0,0}
\definecolor{northwestern-purple}{RGB}{82,0,99}
\definecolor{sauder-green}{RGB}{171,180,0}
\definecolor{saffron}{RGB}{255, 153, 51}

\usepackage[square,numbers,sort&compress]{natbib}

\usepackage{enumerate}

\newcommand{\bb}{\mathbb}
\newcommand{\R}{\bb R}
\newcommand{\Z}{{\bb Z}}
\newcommand{\N}{{\bb N}}
\newcommand{\Q}{{\bb Q}}

\newcommand{\ch}{Chv\'{a}tal }

\DeclareMathOperator\proj{proj}

\DeclareMathOperator*{\aff}{aff}

\DeclareMathOperator*{\intt}{int}

\DeclareMathOperator*{\intcone}{intcone}

\DeclareMathOperator*{\cc}{cc}

\theoremstyle{definition}
\newtheorem{theorem}{Theorem}[section]
\newtheorem{lemma}[theorem]{Lemma}

\newtheorem*{definition}{Definition}
\newtheorem{remark}[theorem]{Remark}

\newtheorem{example}[theorem]{Example}

\title{Mixed-integer linear representability, disjunctions, and variable elimination}

\author{Amitabh Basu\footnote{Applied Mathematics and Statistics, Johns Hopkins University, USA. A. Basu was supported by the NSF grant CMMI1452820.}
\and Kipp Martin\footnote{Booth School of Business, University of Chicago}
\and Christopher Thomas Ryan\footnotemark[2]
\and Guanyi Wang\footnote{Industrial and Systems Engineering, Georgia Institute of Technology}}

\begin{document}

\maketitle

\begin{abstract}
Jeroslow and Lowe gave an exact geometric characterization of subsets of $\R^n$ that are projections of mixed-integer linear sets, a.k.a MILP-representable sets. We give an alternate algebraic characterization by showing that a set is MILP-representable {\em if and only if} the set can be described as the intersection of finitely many {\em affine \ch inequalities}. These inequalities are a modification of a concept introduced by Blair and Jeroslow.
%of MILP-representable sets as exactly those sets that be described by finitely-many inequalities involving the so-called \ch functions.
This gives a sequential variable elimination scheme that, when applied to the MILP representation of a set, explicitly gives the affine \ch inequalities characterizing the set. This is related to the elimination scheme of Wiliams and Williams-Hooker, who describe projections of integer sets using \emph{disjunctions} of \ch systems. Our scheme extends their work in two ways. First, we show that disjunctions are unnecessary, by showing how to find the affine \ch inequalities that cannot be discovered by the Williams-Hooker scheme. Second, disjunctions of \ch systems can give sets that are \emph{not} projections of mixed-integer linear sets; so the Williams-Hooker approach does not give an exact characterization of MILP representability. %In contrast, we show that \ch inequalities without disjunctions {\em exactly} characterize MILP representability: a set can be described as the intersection of finitely many \ch inequalities {\em if and only if} the set is MILP-representable.
\end{abstract}

\section{Introduction}

Understanding which problems can be modeled as a mixed-integer linear program using additional integer and continuous variables is an important question for the discrete optimization community. More precisely, researchers are interested in characterizing sets that are projections of mixed-integer sets described by linear constraints. Such sets have been termed \emph{MILP-representable sets}; see \cite{vielma2015mixed} for a thorough survey. A seminal result of Jeroslow and Lowe \cite{jeroslow1984modelling} provides a geometric characterization of MILP-representable sets as the sum of a finitely generated monoid, and a disjunction of finitely many polytopes (see Theorem~\ref{theorem:jeroslow-lowe} below for a precise statement). %This characterization, although insightful, is limited by its geometric flavor and its use of disjunctions. %Given a MILP-representable set there is no systematic way of providing its representation in terms of these disjunctions. 
%In this paper we provide a constructive algebraic characterization of MILP-representability. 
An algebraic approach based on explicit elimination schemes for integer variables is considered in \cite{williams-1,williams-2,williams-hooker,balas2011projecting}. This prior work tries to adapt the Fourier-Motzkin elimination approach for linear inequalities to handle integer variables. %\cite{williams-hooker} represents a watermark in that direction (following on previous work in \cite{williams-2,williams-1}) by providing a sequential integer variable elimination scheme analogous to Fourier-Motzkin elimination that provides an algebraic approach to projecting mixed-integer sets. 
Balas, in \cite{balas2011projecting}, also explores how to adapt Fourier-Motzkin elimination in the case of binary variables. In both instances, there is a need to introduce \emph{disjunctions} of inequalities that involve either rounding operations or congruence relations. We emphasize that the geometric approach of Jeroslow-Lowe and the algebraic approach of Williams-Hooker both require the use of disjunctions. %Moreover, the approach of \cite{williams-hooker} does not yield a characterization of MILP-representable sets (see Example~\ref{ex:dmic-is-too-big} below), it only provides a language to describe the projection of mixed-integer linear sets. 

Our point of departure is that we provide a constructive algebraic characterization of MILP-representability that does not need disjunctions, but instead makes use of {\em affine \ch inequalities}, i.e. affine linear inequalities with rounding operations (for a precise definition see Definition~\ref{definition:ch-functions} below). We show that MILP-representable sets are exactly those sets that satisfy a finite system of affine \ch inequalities. %Our precise statement is that every MILP-representable set is a mixed-integer set described by finitely many \ch inequalities. 
In contrast, Williams and Hooker~\cite{williams-1,williams-2,williams-hooker} require {\em disjunctions} of systems of affine \ch inequalities. Another disadvantage in their work is the following: there exist sets given by disjunctions of affine \ch systems that are {\em not} MILP-representable.  Finally, our proof of the non-disjunctive characterization is constructive %, involving, on the one hand, swapping rounding operations for additional integer variables (see Section~\ref{ss:guanyi-trick-direction}) and, on the other hand, using Hilbert bases of cones combined with rounding to replace integer variables with \ch inequalities (see Section~\ref{ss:theory-heavy-direction}). The constructive approach
and implies a sequential variable elimination scheme for mixed-integer linear sets (see Section~\ref{sec:sequential-elim}). 

{\em We thus simultaneously show three things: 1) disjunctions are not necessary for MILP-representability (if one allows affine \ch inequalities), an operation that shows up in both the Jeroslow-Lowe and the Williams-Hooker approaches, 2) our algebraic characterization comes with a variable elimination scheme, which is an advantage, in our opinion, to the geometric approach of Jeroslow-Lowe, and 3) our algebraic characterization is exact, as opposed to the algebraic approach of Williams-Hooker, whose algebraic descriptions give a strictly larger family of sets than MILP-representable sets.}

Using our characterization we resolve an open question posed in Ryan \cite{ryan1991} on the representability of integer monoids. Theorem 1 in \cite{ryan1991} shows that every finitely-generated integer monoid can be described as a finite system of \ch inequalities but leaves open the question of how to construct the associated \ch functions via elimination. Ryan states that the elimination methods of Williams in \cite{williams-1,williams-2} do not address her question because of the introduction of disjunctions. Our work provides a constructive approach for finding a \ch inequality representation of finitely-generated integer monoids using elimination. 

\section{Preliminaries}\label{s:preliminaries}

%This section introduces notation, basic definitions and background results for our later development.  
$\Z, \Q, \R$ denote the set of integers, rational numbers and reals, respectively. Any of these sets subscripted by a plus means the nonnegative elements of that set. For instance, $\Q_+$ is the set of nonnegative rational numbers. The projection operator $\proj_Z$ where $Z \subseteq \left\{x_1, \dots, x_n\right\}$ projects a vector $x \in \R^n$ onto the coordinates in $Z$. Following \cite{jeroslow1984modelling} we say a set $S \in \R^n$ is \emph{mixed integer linear representable} (or MILP-representable) if there exists rational matrices $A, B, C$ and a rational vector $d$ such that 
\begin{align}\label{eq:milp-represented}
S = \proj_x \left\{ (x,y,z) \in \R^n \times \R^p \times \Z^q : Ax + By + Cz \ge d\right\}.
\end{align}

Let $S$ be a finite subset of vectors in $\R^n$. The set of nonnegative integer linear combinations is denoted $\intcone A$. The following is the main result from~\cite{jeroslow1984modelling} stated as Theorem 4.47 in \cite{conforti2014integer}:

\begin{theorem}\label{theorem:jeroslow-lowe}
A set $S \subset \R^n$ is MILP-representable if and only if there exists rational polytopes $P_1, \dots, P_k \subseteq \R^n$ and vectors $r^1, \dots, r^k \in \Z^n$ such that 
\begin{align}\label{eq:jeroslow-lowe-chraracterization}
S = \bigcup_{i=1}^k P_i + \intcone \left\{r^1, \dots, r^t\right\}.
\end{align}
\end{theorem}

The ceiling operator $\lceil a \rceil$ gives the smallest integer no less than $a \in \R$. {\em \ch functions}, first introduced by \cite{blair82}, are obtained by taking linear combinations of linear functions and using the ceiling operator. We extend this original definition to allow for affine linear functions, as opposed to homogenous linear functions. Consequently, we term our functions {\em affine \ch functions}. We use the concept of finite binary trees from \cite{ryan1986thesis} to formally define these functions.

\begin{definition}\label{definition:ch-functions}
An affine \ch function $f:\R^n\to \R$ is constructed as follows. We are given a finite binary tree where each node of the tree is either: (i) a leaf, which corresponds to an affine linear function on $\R^n$ with rational coefficients; (ii) has one child with corresponding edge labeled by either a $\lceil \cdot \rceil $ or a number in $\Q_+$, or (iii) has two children, each with edges labelled by a number in $\Q_+$.  

The function $f$ is built as follows. Start at the root node and (recursively) form functions corresponding to subtrees rooted at its children. If the root has a single child whose subtree is $g$, then either (a) $f = \lceil g \rceil $ if the corresponding edge is labeled $\lceil \cdot \rceil $ or (b) $f = \alpha g$ if the corresponing edge is labeled by $a \in \Q_+$.
If the root has two children with corresponding edges labeled by $a \in \Q_+$ and $b \in \Q_+$ then $f = ag + bh$ where $g$  and $h$  are  functions corresponding to the respective children of the root.\footnote{The original definition of \ch function in \cite{blair82} does not employ binary trees. Ryan shows the two definitions are equivalent in~\cite{ryan1986thesis}.}

The \emph{depth} of a binary tree representation $T$ of an affine \ch function is the length of the longest path from the root to a node in $T$, and $\cc(T)$ denotes the {\em ceiling count of $T$}, i.e., the total number of edges of $T$ labeled $\lceil \cdot \rceil$.
\end{definition}

The original definition of {\em \ch function} in the literature requires the leaves of the binary tree to be linear functions, and the domain of the function to be $\Q^n$ (see \cite{ryan1986thesis,blair82,ryan1991}). Our definition above allows for {\em affine} linear functions at the leaves, and the domain of the functions to be $\R^n$. We use the term {\em \ch function}, as opposed to {\em affine} \ch function, to refer to the setting where the leaves are linear functions. In this paper, the domain of all functions is $\R^n$. 

%Blair and Jeroslow introduced the term {\em \ch function} to refer to the setting when all leaves of the binary tree are linear functions, as opposed to affine linear. We will mean the same thing when we say \ch function in this paper. %When there is no need to make a distinction or it is clear from context, we will simply say \ch function. 
%Traditionally, \ch functions were defined on the restricted domain $\Q^n$, for some natural number $n$ (see \cite{ryan1986thesis,blair82,ryan1991}). We do not make this restriction. Our definition restricts all coefficients to be rational; thus, for every affine \ch function, $f(\Q^n) \subseteq \Q$.

An inequality $f(x) \leq b$, where $f$ is an affine \ch function and $b\in \R$, is called an {\em affine \ch inequality}. %Allowing for non-rational $b$ is important to our development. 
A \emph{mixed-integer \ch (MIC) set} is a mixed-integer set described by finitely many affine \ch inequalities. That is, a set $S$ is a mixed integer \ch set if there exist affine \ch functions $f_i$ and $b_i \in \R$ for $i = 1, \dots, m$ such that
$S = \{ (x,z) \in \R^n \times \Z^q : f_i(x,z) \le b_i  \text{ for } i = 1, \dots, m \}. $
A set $S$ is a \emph{disjunctive mixed-integer \ch (DMIC) set} if there exist affine \ch functions $f_{ij}$ and $b_{ij} \in \R$ for $i = 1, \dots, m$ and $j = 1, \dots, t$ such that $S = \bigcup_{j = 1}^t \{ (x,z) \in \R^n \times \Z^q : f_{ij}(x,z) \le b_{ij}  \text{ for } i = 1, \dots, m \}.$

\section{MILP-representable sets as DMIC sets}\label{s:cgm}

From the perspective of MILP-representability, the following result summarizes the work in~\cite{williams-2,williams-1,williams-hooker} that relates affine \ch functions and projections of integer variables.

\begin{theorem}\label{theorem:w-h-gives-projection}
Every MILP-representable set is a DMIC set.
\end{theorem}

Theorem~\ref{theorem:w-h-gives-projection} is not explicitly stated in \cite{williams-2,williams-1,williams-hooker}, even though it summarizes the main results of these papers, for two reasons: (i) the development in \cite{williams-hooker} works with linear congruences and inequalities as constraints and not affine \ch inequalities and (ii) they only treat the pure integer case. These differences are only superficial. For (i), an observation due to Ryan in~\cite{ryan1991} shows that congruences can always be expressed equivalently as affine \ch inequalities. For (ii), continuous variables (the $y$ variables in \eqref{eq:milp-represented}) can first be eliminated using Fourier-Motzkin elimination, which introduces no complications.

The converse of Theorem~\ref{theorem:w-h-gives-projection} is not true. As the following example illustrates, not every DMIC set is MILP-representable. 

\begin{example}\label{ex:dmic-is-too-big}
Consider the set $E := \{(\lambda, 2\lambda): \lambda \in \Z_+\} \cup \{(2\lambda, \lambda): \lambda \in \Z_+\}$ as illustrated in Figure~\ref{fig:counter-example}. 
\begin{figure}
\begin{center}
\includegraphics[scale=.75]{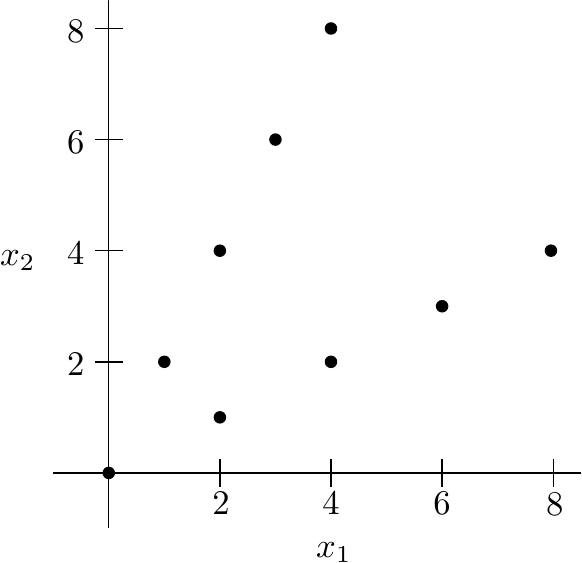}
\end{center}
\caption{A DMIC set that is not MILP-representable.}
\label{fig:counter-example}
\end{figure}
This set is a DMIC set because it can be expressed as $E =\{x \in \Z^2_+ : 2x_1 - x_2 = 0\} \cup \{x \in \Z^2_+ : x_1 - 2x_2 = 0\}.$

$E$ is not the projection of any mixed integer linear program. Indeed, by Theorem~\ref{theorem:jeroslow-lowe} every MILP-representable set has the form \eqref{eq:jeroslow-lowe-chraracterization}. Suppose $E$ has such a form. Consider the integer points in $E$ of the form $(\lambda, 2 \lambda)$ for $\lambda \in \Z_+$. There are infinitely many such points and so cannot be captured inside of the finitely-many polytopes $P_k$ in \eqref{eq:jeroslow-lowe-chraracterization}. Thus, the ray $\lambda (1, 2)$ for $\lambda \in \Z_+$ must lie inside $\intcone \{r^1, \dots, r^t\}$. Identical reasoning implies the ray $\lambda (2,1)$ for $\lambda \in \Z_+$ must also lie inside $\intcone \{r^1, \dots, r^t\}$. But then, every conic integer combination of these two rays must lie in $E$. Observe that $(3,3) = (2,1) + (1,2)$ is one such integer combination but $(3,3) \notin E$. We  conclude that $E$ cannot be represented in the form \eqref{eq:jeroslow-lowe-chraracterization} and hence $E$ is not MILP-representable.
\end{example}

\section{Characterization of MILP-representable sets as MIC sets}\label{s:MILP-as-MIC}

In this section we characterize MILP-representable sets as MIC sets. This is achieved in two steps across two subsections. The main results are Theorems~\ref{theorem:mic-is-milp} and~\ref{theorem:milp-is-mic}, which are converses of each other.

\subsection{MIC sets are MILP-representable}\label{ss:guanyi-trick-direction}

We show how to ``lift'' a MIC set to a mixed-integer linear set. The idea is simple -- replace ceiling operators with additional integer variables. However, we need to work with an appropriate representation of an affine \ch function in order to implement this idea. The next result provides the correct representation.

\begin{theorem} \label{theorem:PossibleChvatal}
For every affine \ch function $f$ represented by a binary tree $T$, one of the following cases hold:
\begin{description}
 \item[Case 1:] $\cc(T) = 0$, which implies that $f$ is an affine linear function.
 \item[Case 2:] $f = \gamma \lceil g_1 \rceil + g_2$, where $\gamma >0$ and $g_1, g_2$ are affine \ch functions such that there exist binary tree representations $T_1, T_2$ for $g_1, g_2$ respectively,  with $\cc(T_1) + \cc(T_2) + 1 \leq \cc(T)$.
 \end{description}
\end{theorem}

\begin{proof}
We use induction on the depth of the binary tree $T$. For the base case, if $T$ has depth $0$, then $\text{cc}(T) = 0$ and we are in Case 1. The inductive hypothesis assumes that for some $k \geq 0$, every affine \ch function $f$ with a binary tree representation  $T$ of depth less or equal to $k$, can be expressed in Case 1 or 2.

For the inductive step, consider an affine \ch function $f$ with a binary tree representation $T$ of depth $k + 1$. If the root node of $T$ has a single child, let $T'$ be the subtree of $T$ with root node equal to the child of the root node of $T$. We now consider two cases: the edge at the root node is labeled with a $\lceil\cdot\rceil$, or the edge is labeled with a scalar $\alpha > 0$. In the first case, $f = \lceil g \rceil$ where $g$ is an affine \ch function which has $T'$ as a binary tree representation. Also, $\cc(T') + 1 = \cc(T)$. Thus, we are done by setting $g_1 = g$, $g_2 = 0$ and $\gamma = 1$. In the second case, $f = \alpha g$ where $g$ is an affine \ch function which has $T'$ as a binary tree representation, with $\cc(T') = \cc(T)$. Note that $T'$ has smaller depth than $T$. Thus, we can apply the induction hypothesis on $g$ with representation $T'$. If this ends up in Case 1, then $0 = \cc(T') = \cc(T)$ and $f$ is in Case 1. Otherwise, we obtain $\gamma' > 0$, affine \ch functions $g'_1, g'_2$, and binary trees $T'_1, T'_2$ representing $g'_1, g'_2$ respectively, with \begin{equation}\label{eq:cc-single-edge}\cc(T'_1) + \cc(T'_2) + 1 \leq \cc(T') = \cc(T)\end{equation} such that $g = \gamma'\lceil g'_1 \rceil + g'_2$. Now set $\gamma = \alpha \gamma'$, $g_1 = g'_1$, $g_2 = \alpha g'_2$, $T_1 = T'_1$ and $T_2$ to be the tree whose root node has a single child with $T'_2$ as the subtree, and the edge at the root labeled with $\alpha$. Note that $\cc(T_2) = \cc(T'_2)$. Also, observe that $T_1, T_2$ represents $g_1, g_2$ resepectively. Combined with~\eqref{eq:cc-single-edge}, we obtain that $\cc(T_1) + \cc(T_2) + 1 \leq \cc(T)$.

If the root node of $T$ has two children, let $S_1, S_2$ be the subtrees of $T$ with root nodes equal to the left and right child, respectively, of the root node of $T$. Then, $f = \alpha h_1 + \beta h_2$, where $\alpha, \beta>0$ and $h_1, h_2$ are affine \ch functions with binary tree representations $S_1, S_2$ respectively. Also note that the depths of $S_1, S_2$ are both strictly less than the depth of $T$, and 

\begin{equation}\label{eq:cc-h1-h2-f}\cc(S_1) + \cc(S_2) = \cc(T)\end{equation}

By the induction hypothesis applied to $h_1$ and $h_2$ with representations $S_1, S_2$, we can assume both of them end up in Case 1 or 2 of the statement of the theorem. If both of them are in Case 1, then $\cc(S_1) = \cc(S_2) = 0$, and by~\eqref{eq:cc-h1-h2-f}, $\cc(T) = 0$. So $f$ is in Case 1. 

Thus, we may assume that $h_1$ or $h_2$ (or both) end up in Case 2. There are three subcases, (i) $h_1, h_2$ are both in Case 2, (ii) $h_1$ is Case 2 and $h_2$ in Case 1, or (iii) $h_2$ in Case 2 and $h_1$ in Case 1. We analyze subcase (i), the other two subcases are analogous. This implies that there exists $\gamma'>0$, and affine \ch functions $g'_1$ and $g'_2$ such that $h_1 = \gamma' \lceil g'_1 \rceil + g'_2$, and there exist binary tree representations $T'_1, T'_2$ for $g'_1, g'_2$ respectively, such that \begin{equation}\label{eq:cc-g'1-g'2-h1}\cc(T'_1) + \cc(T'_2) + 1 \leq \cc(S_1).\end{equation} Now set $\gamma = \alpha\gamma'$, $g_1(x) = g'_1(x)$ and $g_2(x) = \alpha g'_2(x) + \beta h_2(x)$. Then $f = \gamma\lceil g_1 \rceil + g_2$. Observe that $g_2$ has a binary tree representation $T_2$ such that the root node of $T_2$ has two children: the subtrees corresponding to these children are $T'_2$ and $S_2$, and the edges at the root node of $T_2$ are labeled by $\alpha$ and $\beta$ respectively. Therefore, \begin{equation}\label{eq:cc-g2-g'2-h2}\cc(T_2) \leq \cc(T'_2) + \cc(S_2).\end{equation}
Moreover, we can take $T_1 = T'_1$ as the binary tree representation of $g_1$. We observe that \begin{align*}\begin{array}{rcl} \cc(T_1) + \cc(T_2) + 1 &\leq &\cc(T'_1) + \cc(T'_2) + \cc(S_2) + 1 \\ & \leq & \cc(S_1) + \cc(S_2) =  \cc(T)\end{array}\end{align*} where the first inequality is from the fact that $T_1 = T'_1$ and~\eqref{eq:cc-g2-g'2-h2}, the second inequality is from~\eqref{eq:cc-g'1-g'2-h1} and the final equation is~\eqref{eq:cc-h1-h2-f}.\end{proof}

For a system of affine \ch inequalities where each affine \ch function is represented by a binary tree, the \emph{total ceiling count of this representation} is the sum of the ceiling counts of all these
binary trees. The next lemma shows how to reduce the ceiling count of a MIC set by one, in exchange for an additional integer variable. 

\begin{lemma}\label{lem:reduce-ceil-count} 
Given a system $C=\{x \in \R^n\times\mathbb{Z}^q: ~ f_i(x) \leq b_i\}$ of affine \ch inequalities with a total ceiling count $c \geq 1$, there exists a system $P=\{(x,z) \in \R^n\times\mathbb{Z}^q \times \Z: ~ f'_i(x) \leq b'_i\}$ of affine \ch inequalities with a total ceiling count of at most $c-1$, and $C = \proj_{x}(P)$.
\end{lemma}

\begin{proof} Since $c\geq 1$, at least one of the $f_i$ is given with a binary tree representation $T$ with strictly positive ceiling count. Without loss of generality we assume it is $f_1$. This means $f_1$, along with its binary tree representation $T$, falls in Case 2 of Theorem~\ref{theorem:PossibleChvatal}. %since linear functions have ceiling count $0$. 
Therefore, one can write $f$ as $f_1 = \gamma \lceil g_1 \rceil + g_2,$ with $\gamma >0$, and $g_1, g_2$ are affine \ch functions such that there exist binary tree representations $T_1, T_2$ for $g_1, g_2$ respectively,  with $\cc(T_1) + \cc(T_2) + 1 \leq \cc(T)$. Dividing by $\gamma$ on both sides, the inequality $f_1(x) \leq b_1$ is equivalent to $
        \lceil g_1(x) \rceil +(1/\gamma) g_2(x) \leq b_1/\gamma.$ Moving $(1/\gamma) g_2(x)$ to the right hand side, we get $
        \lceil g_1(x) \rceil \leq -(1/\gamma) g_2(x) + b_1/\gamma.$
      This inequality is easily seen to be equivalent to two inequalities, involving an extra integer variable $z \in \mathbb{Z}$: $
        \lceil g_1(x) \rceil \leq z \leq - (1/\gamma) g_2(x) + b_1 / \gamma, \label{eq:with-ceiling-case1}$
      which, in turn is equivalent to $ g_1(x) \leq z \leq -(1/\gamma) g_2(x) + b_1 / \gamma,$ since $z \in \mathbb{Z}$. Therefore, we can replace the constraint $f_1(x) \leq b_1$ with the two constraints      
      \begin{align}
        & g_1(x) - z \leq 0, \label{eq:g1}\\
        & (1/\gamma) g_2(x) + z \leq b_1/ \gamma  \Leftrightarrow g_2(x) + \gamma z \leq b_1\label{eq:g2}
      \end{align}
as long as we restrict $z \in \Z$. Note that the affine \ch functions on the left hand sides of~\eqref{eq:g1} and~\eqref{eq:g2} have binary tree representations with ceiling count equal to $\cc(T_1)$ and $\cc(T_2)$ respectively. Since $\cc(T_1) + \cc(T_2) + 1 \leq \cc(T)$, the total ceiling count of the new system is at least one less than the total ceiling count of the previous system.\end{proof}

The key result of this subsection is an immediate consequence.

\begin{theorem}\label{theorem:mic-is-milp}
Every MIC set is MILP-representable. 
\end{theorem}

\begin{proof} Consider any system of affine \ch inequalities describing the MIC set, with total ceiling count $c \in \N$. Apply Lemma~\ref{lem:reduce-ceil-count} at most $c$ times to get the desired result.\end{proof}

\subsection{MIP-representable sets are MIC sets}\label{ss:theory-heavy-direction}

We now turn to showing the converse of Theorem~\ref{theorem:mic-is-milp}, that every MILP-representable set is a MIC set (Theorem~\ref{theorem:milp-is-mic} below). This direction leverages some established theory in integer programming, in particular,

\begin{theorem}\label{theorem:b-j-mod}
For any rational $m \times n$ matrix $A$, there exists a finite set of Chv\'atal functions $f_i: \R^m \to \R$, $i\in I$ with the following property: for every $b\in \R^m$, $\{ x \in \Z^n  \, : \, Ax \ge b\}$ is nonempty if and only if $f_i(b) \leq 0$ for all $i\in I$. Moreover, these functions can be explicitly constructed from the matrix $A$.
\end{theorem}

The above result is quite similar to Corollary 23.4 in~\cite{schrijver86} (see also Theorem~5.1 in \cite{blair82}). The main difference here is that we allow the right hand side $b$ to be non rational. This difference is indispensable in our analysis (see the proof of Theorem~\ref{theorem:milp-is-mic}). Although our proof of Theorem~\ref{theorem:b-j-mod} is conceptually similar to the approach in~\cite{schrijver86}, we need to handle some additional technicalities related to irrationality. We relegate this analysis to the appendix (Section~\ref{ss:proof-of-claim-2}). %Now for the main result of this subsection. 
The following lemma is easy to verify.

\begin{lemma}\label{lem:affine-comp-chvatal} 
Let $T:\R^{n_1} \to \R^{n_2}$ be an affine transformation involving rational coefficients, and let $f:\R^{n_2} \to \R$ be an affine \ch function. Then $f\circ T :\R^{n_1} \to \R$ can be expressed as $f\circ T(x) = g(x)$ for some affine \ch function $g : \R^{n_1}\to \R$.
\end{lemma}

%Now for the main result of this subsection.

\begin{theorem}\label{theorem:milp-is-mic}
Every MILP-representable set is a MIC set.
\end{theorem}

\begin{proof} 
Let $m, n,  p, q \in \N$. Let $A \in \Q^{m \times n}, B \in \Q^{m\times p}, C \in \Q^{m\times q}$ be any rational matrices, and let $d\in \Q^m$. Define $
\mathcal{F} = \{(x,y,z) \in \R^n \times \R^p \times \Z^q \;: \;\;Ax + By + Cz \geq d\}.$
It suffices to show that the projection of $\mathcal{F}$ onto the $x$ space is a MIC set.

By applying Fourier-Motzkin elimination on the $y$ variables, we obtain rational matrices $A', C'$ with $m'$ rows for some natural number $m'$, and a vector $d'\in \Q^{m'}$ such that the projection of $\mathcal{F}$ onto the $(x,z)$ space is given by $\overline{\mathcal{F}}:= \{(x,z) \in \R^n \times \Z^q \;: \;\;A'x + C'z \geq d'\}.$

Let $f_i:\R^{m'}\to \R$, $i\in I$ be the set of Chv\'atal functions obtained by applying Theorem~\ref{theorem:b-j-mod} to the matrix $C'$. %Define $G:\R^{q} \to \R^{m'}$ as the function $G(y) : = (g_1(y), \ldots, g_m(y))$. 
It suffices to show that the projection of $\overline{\mathcal{F}}$ onto the $x$ space is $\hat{\mathcal{F}}:=\{x\in \R^n \;: \; f_i(d' - A'x) \leq 0, \;\; i\in I\}$
since for every $i\in I$, $ f_i(d' - A'x) \leq 0$ can be written as $g_i(x) \leq 0$ for some affine Chv\'atal function $g_i$, by Lemma~\ref{lem:affine-comp-chvatal}.\footnote{This is precisely where we need to allow the arguments of the $f_i$'s to be non rational because the vector $d' - A'x$ that arise from all possible $x$ is sometimes non rational.} This follows from the following sequence of equivalences.
\begin{align*}
\begin{array}{rcl}
x \in \proj_x(\mathcal{F}) & \Leftrightarrow & x \in \proj_x(\overline{\mathcal{F}}) \\
& \Leftrightarrow & \exists z \in \Z^{q} \textrm{ such that } (x,z) \in \overline{\mathcal{F}} \\
& \Leftrightarrow & \exists z \in \Z^{q}  \textrm{ such that } C'z \geq d' -A'x \\ 
& \Leftrightarrow & f_i(d' -A'x) \leq 0 \textrm{ for all } i\in I \qquad (\textrm{By Theorem~\ref{theorem:b-j-mod}}) \\
& \Leftrightarrow & x \in \hat{\mathcal{F}}. \qquad (\textrm{By definition of } \hat{\mathcal{F}})\qquad
\end{array} 
\end{align*}
\end{proof}

\begin{remark}\label{rem:homogeneous} We note in the proof of Theorem~\ref{theorem:milp-is-mic} that if the right hand side $d$ of the mixed-integer set is $0$, then the affine \ch functions $g_i$ are actually \ch functions. This follows from the fact that the function $g$ in Lemma~\ref{lem:affine-comp-chvatal} is a \ch function if $f$ is a \ch function and $T$ is a linear transformation.\end{remark}

\section{An sequential variable elimination scheme for mixed-integer \ch sytems}\label{sec:sequential-elim}

Ryan shows (see Theorem 1 in~\cite{ryan1991}) that $Y$ is a finitely generated integral monoid if and only if there exist \ch functions $f_{1}, \ldots, f_p$ such that $
Y = \{ b  \, : \, f_{i}(b) \leq 0, \, i = 1, \ldots, p\}.$ By definition, a finitely generated integral monoid $Y$ is MILP representable since $Y = \{  b \, :  \, b = Ax,  x \in \Z^{n}_{+}\}$ where $A$ is an $m \times n$ integral matrix.  Thus, an alternate proof of Ryan's characterization follows from Theorems~\ref{theorem:mic-is-milp} and~\ref{theorem:milp-is-mic} and Remark~\ref{rem:homogeneous}. 

Ryan~\cite{ryan1991}  further states that  ``It is an interesting open problem to find an elimination scheme to construct the \ch constraints for an arbitrary finitely generated integral monoid.''  The results of Section 4 provide such an elimination scheme, as we show below. 

A number of authors have studied sequential projection algorithms for  linear integer programs~\cite{williams-1,williams-2,williams-hooker,balas2011projecting}. However, their sequential projection algorithms {\it do not} resolve Ryan's open question because they do not generate the \ch functions $f_{i}(b)$ required to describe $Y$. Below, we show that, in fact, all these schemes have to {\em necessarily} resort to the use of disjunctions because they try to adapt the classical Fourier-Motzkin procedure and apply it to the system $b = Ax,  x \in \Z^{n}_{+}$.

Our resolution to Ryan's open question hinges on the observation that the \ch functions that define $Y$ can be generated if certain redundant linear inequalities are added to those generated by the  Fourier-Motzkin procedure, and then the ceiling operator is applied to these redundant inequalities.  We illustrate the idea with Example~\ref{example:fm-fails} below and then outline the general procedure. Rather than work with $Ax = b, $ $x \in \Z^n_{+}$ we work with the system $Ax \ge b$ and $x \in \Z^n$.
\begin{example}\label{example:fm-fails}
Let $\mathcal{B}$ denote= the set of all $b = (b_1, \ldots, b_5) \in \R^5$ such that there exist $x_1, x_2, x_3 \in \Z$ satisfying the following inequalities. 
\begin{align}\label{eq:fm-example}
\begin{array}{rrrcr}
- x_1  &+\frac{1}{2} x_2 &-\frac{1}{10} x_3& \ge & b_1 \\
x_1 &-\frac{1}{4} x_2&&\ge& b_2\\
&-x_2&+x_3&\ge& b_3\\
&&x_3&\ge& b_4\\
&&-x_3&\ge& b_5
\end{array} 
\end{align}

Performing Fourier-Motzkin elimination on the linear relaxation of \eqref{eq:fm-example} gives 
\begin{align}
      0 &\ge  2b_1 + 2 b_2 + \tfrac{1}{2} b_3 + \tfrac{3}{10} b_5 \label{eq:fm-ray1} \\
      0 & \ge \tfrac{1}{10}b_4 + \tfrac{1}{10}b_5. \label{eq:fm-ray2}
\end{align}
Unfortunately, there is no possible application of the ceiling operator to any combination of terms in these two inequalities that results in \ch functions that characterize $\mathcal{B}$. In particular, $b^1 = (0,0,0,1,-1) \notin \mathcal{B}$ while $b^2 = (-1, 0,0,1,-1) \in \mathcal{B}$. 
Consider $b^1$. This forces $x_3 = 1$  and the only feasible values for $x_1$ are $ 1/10 \le x_1 \le 4/10.$  Therefore, for this set of $b$ values applying the ceiling operator to some combination of terms in~\eqref{eq:fm-ray1}-~\eqref{eq:fm-ray2}  must result in either~\eqref{eq:fm-ray1} positive or~\eqref{eq:fm-ray2}   positive.  Since $b_1 = b_2 = b_3 = 0$  and $b_5 = -1$ there is no ceiling operator that can be applied to any term in~\eqref{eq:fm-ray1} to make  the right hand side positive.  Hence a ceiling operator must be applied to~\eqref{eq:fm-ray2} in order to make the right hand side positive for $b_4 = 1$ and $b_5 = -1.$  However, consider $b^2$. For this right-hand-side, $x_{1}  = x_{2} = x_{3} = 1$ is feasible.  Since we still have  $b_4 = 1$ and $b_5 = -1$, the ceiling operator applied to~\eqref{eq:fm-ray2} will incorrectly conclude that there is no integer solution with $b^2$.%when  $b_1 = -1$, $b_2 = b_3 = 0,$ $b_4 = 1,$ and $b_5 = -1.$ 

% In Appendix~\ref{sec:example-2-analysis} we rigorously argue that no combination of ceiling operators on the two inequalities can separate $b^1$ and $b^2$.

However, Fourier-Motzkin elimination \emph{will} work in conjunction with ceiling operations if {\em appropriate redundant inequalites are added.} Consider the  inequality $x_1  \ge  b_1 + 2 b_2 + \frac{1}{10} b_4$
  which is redundant to~\eqref{eq:fm-example}. Integrality of $x_1$ implies  $ x_1  \ge  \lceil b_1 + 2 b_2 + \frac{1}{10} b_4 \rceil.$
 Applying Fourier-Motzkin elimination to~\eqref{eq:fm-example} along with  $ x_1  \ge  \lceil b_1 + 2 b_2 + \frac{1}{10} b_4 \rceil$  generates  the additional inequality $0   \ge b_1 + \frac{1}{2} b_3 + \lceil b_1 + 2 b_2 + \frac{1}{10} b_4 \rceil + \frac{4}{10} b_5$, which separates $b^1$ and $b^2$.
\end{example}

%  Now applying Fourier -Motzkin elimination to ~\eqref{eq:fm-example} plus this additional inequality gives the following new inequality: $0  \ge b_1 + \frac{1}{2} b_3 + b_1 + 2 b_2 + \frac{1}{10} b_4 + \frac{4}{10} b_5$. Applying the ceiling operators to the terms in the middle, we obtain $
%       0  \ge b_1 + \frac{1}{2} b_3 + \lceil b_1 + 2 b_2 + \frac{1}{10} b_4 \rceil + \frac{4}{10} b_5,$ which separates the two points considered above.

Our proofs in Section~\ref{s:MILP-as-MIC} give a general method to systematically add the necessary redundant constraints, such as $x_1  \ge  b_1 + 2 b_2 + \frac{1}{10} b_4$ in Example~\ref{example:fm-fails}. This results in the following variable elimination scheme: at iterative step $k$ maintain a MIC set with variables $(x_k, \ldots, x_n)$ that is indeed the true projection of the original set onto these variables. By Theorem~\ref{theorem:mic-is-milp}, this MIC set is MILP representable with a set of variables $(x_k, \ldots, x_n, z).$  Then by Theorem~\ref{theorem:milp-is-mic} we can project out variable $x_k$ and additional auxiliary $z$ variables that were used to generate the MILP representation and obtain a new  MIC in only variables $(x_{k+1}, \ldots, x_n).$ {\em The key point is that adding these auxiliary variables and then using Theorem~\ref{theorem:milp-is-mic} introduces the necessary redundant inequalities.} Repeat until all variables are eliminated and a MIC set remains in the $b$ variables.  This positively answers the question of Ryan~\cite{ryan1991} and provides a projection algorithm in a similar vein to Williams \cite{williams-1,williams-2,williams-hooker}  and Balas \cite{balas2011projecting} but {\it without} use of disjunctions.

\bibliographystyle{plain}
\bibliography{../../../../references/references}

%%%%%%%%%%%%%%%%%%%%%%%%%%%%%%%%%%%%%%%%%%%%%%%%%%%%%%%%%%%%%%%%%%%%%%%%%%%%%%%%%
%%%%%%%%%%%%%%%%%%%%%%%%%%%%%%%%%%%%%%%%%%%%%%%%%%%%%%%%%%%%%%%%%%%%%%%%%%%%%%%%%
%%%%%%%%%%%%%%%%%%%%%%%%%%%%%%%%%%%%%%%%%%%%%%%%%%%%%%%%%%%%%%%%%%%%%%%%%%%%%%%%%
%%%%%%%%%%%%%%%%%%%%%%%%%%%%%%%%%%%%%%%%%%%%%%%%%%%%%%%%%%%%%%%%%%%%%%%%%%%%%%%%%
%%%%%%%%%%%%%%%%%%%%%%%%%%%%%%%%%%%%%%%%%%%%%%%%%%%%%%%%%%%%%%%%%%%%%%%%%%%%%%%%%
%%%%%%%%%%%%%%%%%%%%%%%%%%%%%%%%%%%%%%%%%%%%%%%%%%%%%%%%%%%%%%%%%%%%%%%%%%%%%%%%%
%%%%%%%%%%%%%%%%%%%%%%%%%%%%%%%%%%%%%%%%%%%%%%%%%%%%%%%%%%%%%%%%%%%%%%%%%%%%%%%%%
%%%%%%%%%%%%%%%%%%%%%%%%%%%%%%%%%%%%%%%%%%%%%%%%%%%%%%%%%%%%%%%%%%%%%%%%%%%%%%%%%
%%%%%%%%%%% APPENDIX %%%%%%%%%%%%%%%%%%%%%%%%%%%%%%%%%%%%%%%%%%%%%%%%%%%%%%%%%%%%
%%%%%%%%%%%%%%%%%%%%%%%%%%%%%%%%%%%%%%%%%%%%%%%%%%%%%%%%%%%%%%%%%%%%%%%%%%%%%%%%%
%%%%%%%%%%%%%%%%%%%%%%%%%%%%%%%%%%%%%%%%%%%%%%%%%%%%%%%%%%%%%%%%%%%%%%%%%%%%%%%%%
%%%%%%%%%%%%%%%%%%%%%%%%%%%%%%%%%%%%%%%%%%%%%%%%%%%%%%%%%%%%%%%%%%%%%%%%%%%%%%%%%
%%%%%%%%%%%%%%%%%%%%%%%%%%%%%%%%%%%%%%%%%%%%%%%%%%%%%%%%%%%%%%%%%%%%%%%%%%%%%%%%%
%%%%%%%%%%%%%%%%%%%%%%%%%%%%%%%%%%%%%%%%%%%%%%%%%%%%%%%%%%%%%%%%%%%%%%%%%%%%%%%%%
%%%%%%%%%%%%%%%%%%%%%%%%%%%%%%%%%%%%%%%%%%%%%%%%%%%%%%%%%%%%%%%%%%%%%%%%%%%%%%%%%
%%%%%%%%%%%%%%%%%%%%%%%%%%%%%%%%%%%%%%%%%%%%%%%%%%%%%%%%%%%%%%%%%%%%%%%%%%%%%%%%%
%%%%%%%%%%%%%%%%%%%%%%%%%%%%%%%%%%%%%%%%%%%%%%%%%%%%%%%%%%%%%%%%%%%%%%%%%%%%%%%%%

\newpage

\appendix

\section{Proof of Theorem~\ref{theorem:b-j-mod}}\label{ss:proof-of-claim-2}

Theorem~\ref{theorem:b-j-mod} is a generalization of Collary 23.4b(i) in \cite{schrijver86} (see also Theorem~5.1 in \cite{blair82}). In previous work, the right-hand side $b$ was assumed to be rational. This, however, is too strong for our purposes. This section proceeds by showing that the supporting results used to prove Corollary 23.4b(i) in \cite{schrijver86}, can be extended to the case where $b$ is non rational. To our knowledge, no previous work has explicitly treated the case where $b$ is non rational.% although the existing proof ideas largely extend to this case. In order to self-contained, careful details are provided in this appendix.

First we need some preliminary definitions and results. A system of linear inequalities $Ax \ge b$ where $A = (a_{ij})$ has $a_{ij} \in \Q$ for all $i,j$ (that is, $A$ is rational) is \emph{totally dual integral} (TDI) if the minimum in the LP-duality equation
\begin{align*}
\min \left\{c^\top x : Ax \ge b \right\} = \max \left\{y^\top b : A^\top y = c, y \ge 0\right\}  
\end{align*}
has an integral optimal solution $y$ for every integer vector $c$ for which the minimum is finite. Note that rationality of $b$ is not assumed in this definition. When $A$ is rational, the system $Ax \ge b$ can be straightforwardly by manipulated so that all coefficients on the $x$ on the right-hand side are integer. Thus, we may often assume without loss that $A$ is integral. 

For our purposes, the significance of a system being TDI is explained by the following result. For any polyhedron $P \subseteq \R^n$, $P'$  denotes its \ch closure. We also recursively define the $t$-th \ch closure of $P$ as $P^{(0)} := P$, and $P^{(t+1)} = (P^{(t)})'$ for $i \geq 1$.

\begin{theorem}[See \cite{schrijver86}  Theorem 23.1]\label{theorem:schrijver-23.1}
Let $P = \{  x \, : \, Ax \ge b \}$ be nonempty and assume $A$ is integral.  If $Ax \ge b$ is a TDI representation of the polyhedron $P$ then
\begin{eqnarray}
P' = \{  x \, : \, Ax \ge  \lceil b \rceil\}.  \label{eq:tdi-ch-closure}
\end{eqnarray}
\end{theorem}

We now show how to manipulate the system $Ax \ge b$ to result in one that is TDI. The main power comes from the fact that this manipulation depends only on $A$ and works for every right-hand side $b$.

\begin{theorem}\label{theorem:basu}
Let $A$ be a rational $m \times n$ matrix.  Then there exists another nonnegative $q \times m$ rational matrix $U$ such that  for every $b$ the polyhedron $P = \{  x \in \R^{n} \, : \,  Ax \ge b \},$  has a representation $P = \{  x \in \R^{n} \, : \,  Mx \ge b' \}$ where the system $Mx \ge b'$ is TDI and $M = UA,$  $b' = Ub.$
\end{theorem}

\begin{proof}
First construct the matrix $U.$  Let  ${\cal P}(\{1, 2, \ldots, m \})$  denote the power set of  $\{1, 2, \ldots, m\}.$   For each subset of rows $a^i$  of  $A$ with $i  \in S$  where    $S \in {\cal P}(\{1, 2, \ldots, m \}),$   define the cone
\begin{eqnarray}
C(S) := \{  a \in \R^n \, : \,  a = \sum_{i \in S} u_i a^i,  u _i \ge 0,  i \in S \}. \label{eq:defineCS}
\end{eqnarray}
By construction the cone $C(S)$ in~\eqref{eq:defineCS}  is a rational polyhedral cone.  Then by Theorem 16.4 in \cite{schrijver86} there exist integer vectors  $m^{k},$ for $k = k^{S}_{1},  k^{S}_{2} \ldots, k^{S}_{q_{S}}$   that define a  Hilbert basis for this cone.   In this indexing scheme  $q_{S}$ is the cardinality of the set $S.$  We assume  that there are  $q_{S}$ distinct indexes $ k^{S}_{1},  k^{S}_{2} \ldots, k^{S}_{q_{S}}$  assigned to each set $S$ in the power set  ${\cal P}(\{1, 2, \ldots, m \}).$  Since  each $m^{k} \in C(S)$  there is  a nonnegative nonnegative  vector $u^{k}$ that generates $m^{k}$. Without loss  each $u^{k}$ is an $m-$dimensional vector since we can assume a component of zero for each component $u^{k}$ not indexed by  $S.$   Thus   $u^{k} A = m^{k}.$   Define a matrix  $U$ to be the matrix with rows $u^{k}$ for all $k =   k^{S}_{1},  k^{S}_{2} \ldots, k^{S}_{q_{S}}$ and $S \in {\cal P}(\{1, 2, \ldots, m \}).$   Then $M = UA$ is a matrix with rows corresponding to all of the Hilbert bases for the power set of $\{ 1, 2, \ldots, m \}.$  That is, the number of rows in $M$ is $q = \sum_{S \in  {\cal P}(\{1, 2, \ldots, m \})} q_{S}.$
\vskip 7pt
We first show that $M x \ge b'$ is a TDI representation of 
\begin{eqnarray}
P =  \{  x \in \R^{n} \, : \,  Ax \ge b \}  = \{  x \in \R^{n} \, : \,  Mx \ge b' \}.  \label{eq:equal-poly}
\end{eqnarray}
Note that  $\{  x \in \R^{n} \, : \,  Ax \ge b \}$  and $\{  x \in \R^{n} \, : \,  Mx \ge b' \}$  define the same polyhedron since the system of the inequalities $Mx \ge b'$ contains all of the inequalities $Ax \ge b$ (this is because the power set of $\{1, 2, \ldots, m \}$ includes each singleton set) plus additional inequalities that are nonnegative aggregations of inequalities in the system  $Ax \ge b.$  In order to show $Mx \ge b'$ is a TDI representation, assume $c \in \R^n$ is an integral vector and the minimum of 
\begin{eqnarray}\label{eq:tdi}
\max \{ y b' \, : \,  y M = c, y \ge 0 \} 
\end{eqnarray}
is finite. It remains to show there is an integral optimal dual solution to \eqref{eq:tdi}. By linear programming duality $\min \{ c x | M x \ge b' \}$ has an optimal solution $\bar x$ and 
\begin{eqnarray}
\max \{ y b' \, : \,  y M = c, y \ge 0 \} =  \min \{ c x \, : \,  M x \ge b' \}.  \label{eq:tdi-1}
\end{eqnarray}
Then by equation~\eqref{eq:equal-poly}
\begin{eqnarray}
 \min \{ c x \, : \,  M x \ge b' \} =  \min \{ c x \, : \,  Ax \ge b\}.   \label{eq:tdi-2}
\end{eqnarray}
and    $\min \{ c x \, : \,  A x \ge b \}$  also has optimal solution  $\bar x.$ Then again by linear programming duality
\begin{eqnarray}
 \min \{ c x \, : \,  Ax \ge b\}  = \max \{ u b  \,  :  \,  u A = c,  \, u \ge 0\}.
\end{eqnarray}
Let $\bar u$ be an optimal dual solution to $ \max \{ u b  \,  :  \,  u A = c,  \, u \ge 0\}.$  Let $i$  index the strictly positive $\bar u_i$  and define  $S = \{ i \, : \,  \bar u_i > 0 \}.$   By construction of  $M$  there is a subset of rows of $M$ that form a Hilbert basis for $C(S).$  By construction of $C(S),$ $\bar u A = c$ implies $c \in C(S).$  Also, since $\bar u$ is an optimal dual solution,  it must satisfy complementary slackness. That is, $\bar u_{i} > 0$ implies that  $ a^{i} \bar x = b_i.$  Therefore $S$ indexes a set of tight constraints in the system $A \bar x \ge b.$ Consider an arbitrary element  $m^{k}$ of the Hilbert basis associated with the cone $C(S)$. There is a corresponding 
  $m-$vector $u^{k}$ with support in $S$ and 
\begin{align*}
u^{k} A \bar x = \sum_{i \in S} u^{k}_{i}a^i \bar x = \sum_{i \in S} u^{k}_{i} b_i = u^{k} b = b'_{k}.
\end{align*}
Since $u^{k} A = m^{k}$ and $u^{k} b = b'_{k}$ we have 
\begin{align}\label{eq:this-is-useful}
 m^{k} \bar x = b'_{k},   \quad \forall k = k^{S}_{1},  k^{S}_{2} \ldots, k^{S}_{q_{S}}.
\end{align}
As argued above, $c \in C(S)$ and is, therefore, generated by nonnegative integer multiples of the $m^{k}$ for  $k = k^{S}_{1},  k^{S}_{2} \ldots, k^{S}_{q_{S}}.$   That is, there exist nonnegative integers $\bar y_{k}$  such that
\begin{align}\label{eq:write-out-m-k}
c = \sum_{k  =  k^{S}_{1}}^{k^{S}_{q_{S}}} \bar y_{k} m^{k}.
\end{align}

Hence there exists a nonnegative $q$-component integer vector $\bar y$ with support contained in the  set  indexed by  $k^{S}_{1},  k^{S}_{2} \ldots, k^{S}_{q_{S}}$
 such that
\begin{align}\label{eq:write-out-c}
c = \bar y M.
\end{align}
Since $\bar y  \ge 0,$  $\bar y$ is feasible to the left hand side of~\eqref{eq:tdi-1}.   We use \eqref{eq:this-is-useful} and \eqref{eq:write-out-m-k} to show 
\begin{align}\label{eq:zero-duality-gap}
\bar y b' = c \bar x,
\end{align}
which implies that $\bar y$ is an optimal integral dual solution to \eqref{eq:tdi} (since $\bar x$  and  $\bar y$ are  primal-dual feasible), implying the result. 

To show \eqref{eq:zero-duality-gap},  use the fact that the support of $\bar y$ is  contained in the  set  indexed by  $k^{S}_{1},  k^{S}_{2} \ldots, k^{S}_{q_{S}}$ which implies
\begin{eqnarray}
\bar y b' = \sum_{k  =  k^{S}_{1}}^{k^{S}_{q_{S}}} \bar y_{k} b'_{k}.
\end{eqnarray}
Then by \eqref{eq:this-is-useful}  substituting $m^{k} \bar x$ for $b'_{k}$ gives
\begin{eqnarray}
\bar y b' = \sum_{k  =  k^{S}_{1}}^{k^{S}_{q_{S}}} \bar y_{k} b'_{k} =  \sum_{k  =  k^{S}_{1}}^{k^{S}_{q_{S}}} \bar y_{k} m^{k} \bar x.
\end{eqnarray}
Then by \eqref{eq:write-out-m-k} substituting $c$ for  $\sum_{k  =  k^{S}_{1}}^{k^{S}_{q_{S}}} \bar y_{k} m^{k} $  gives 
\begin{eqnarray}
\bar y b' = \sum_{k  =  k^{S}_{1}}^{k^{S}_{q_{S}}} \bar y_{k} b'_{k} =  \sum_{k  =  k^{S}_{1}}^{k^{S}_{q_{S}}} \bar y_{k} m^{k} \bar x = c \bar x.
\end{eqnarray}

This gives  \eqref{eq:zero-duality-gap} and completes the proof.\end{proof}

\begin{remark}
When $S$ is a singleton set, i.e. $S = \{ i\}$ then the corresponding $m^{k}$ for $k = k^{S}_{1}$ may be a scaling of the corresponding $a^{i}.$  However, this does not affect our argument that~\eqref{eq:equal-poly} holds.
\end{remark}

\begin{remark}
Each of the $m^{k}$ vectors that define each Hilbert basis may be assumed to be integer.  Therefore if $A$ is an integer matrix,  $M$ is an integer matrix.  
\end{remark}

Next we will also need a series of results about the interaction of lattices and convex sets.

\begin{definition} Let $V$ be a vector space over $\R$. A {\em lattice} in $V$ is the set of all integer combinations of a linearly independent set of vectors $\{a^1, \ldots,  a^m\}$ in $V.$ The set $\{a^1, \ldots,  a^m\}$ is called the basis of the lattice. The lattice is {\em full-dimensional} if it has a basis that spans $V$.
\end{definition}

\begin{definition} Given a full-dimensional lattice $\Lambda$ in a vector space $V$, a {\em $\Lambda$-hyperplane} is an affine hyperplane $H$ in $V$ such that $H = \aff(H \cap \Lambda)$. This implies that for $V = \R^n$ for $H$ to be a $\Z^n$-hyperplane, $H$ must contain $n$ affinitely independent vectors in $\Z^n$. 
\end{definition}

\begin{definition} Let $V$ be a vector space over $\R$ %with an inner product $\langle \cdot, \cdot \rangle$, 
and let $\Lambda$ be a full-dimensional lattice in $V$. Let $\mathcal{H}_{\Lambda}$ denote the set of all $\Lambda$-hyperplanes that contain the origin. %$\{x \in V: \langle r,x\rangle = 0\}$ is a $\Lambda$-hyperplane. 
Let $C \subseteq V$ be a convex set. Given any $H \in \mathcal{H}_{\Lambda}$, we say that the {\em $\Lambda$-width of $C$ parallel to $H$}, denoted by $\ell(C,\Lambda, V, H)$, is the total number of distinct $\Lambda$-hyperplanes parallel to $H$ that have a nonempty intersection with $C$. %The {\em lattice-width} of $C$ with respect to $\Lambda$ is defined as $\ell(C,\Lambda, V) := \min_{r\in \Lambda^\circ} \ell(C,\Lambda, V, r)$.
The {\em lattice-width} of $C$ with respect to $\Lambda$ is defined as $\ell(C,\Lambda, V) := \min_{H \in \mathcal{H}_{\Lambda}} \ell(C,\Lambda, V, H)$.   
%\kipp{I am having a hard time grasping this definition.  Is there an example we can give to make it easier to understand?}  
\end{definition}

We will need this classical ``flatness theorem" from the geometry of numbers -- see Theorem VII.8.3 on page 317 of~\cite{barvinok}, for example. Since our language is different from Barvinok's language, we give a short proof which can be seen as a translation of Barvinok's proof.

\begin{theorem}\label{thm:flatness} Let $V\subseteq \R^n$ be a vector subspace with $\dim(V) = d$, and let $\Lambda$ be a full-dimensional lattice in $V$. Let $C \subseteq V$ be a compact, convex set. If $C \cap \Lambda = \emptyset$, then $\ell(C,\Lambda, V) \leq d^{5/2}$. 
%\kipp{Barvinok does not state this theorem in terms of $\ell(C,\Lambda, V, r)$  -- I have really having trouble with this $\ell(C,\Lambda, V, r)$ construct.}
\end{theorem}
\begin{proof}
Let $b^1, \ldots, b^d\in V$ be a basis of $\Lambda$, and we identify $V$ with $\R^d$ using coordinates of this lattice basis. Then, the lattice $\Lambda$ becomes $\Z^d$ and the dual lattice $\Lambda^*$ is also $\Z^d$. For any $a\in \R^d$ and let $H_a$ be the linear hyperplane in $V$ orthogonal to the vector $a_1b^1 + \ldots + a_db^d \in V$. We now verify that $H\subseteq V$ is a $\Lambda$-hyperplane containing the origin if and only if there exists $a \in \Z^d$ such that $H = H_a$. 

$H$ is a $\Lambda$-hyperplane containing the origin if and only if it is the linear span of $d-1$ vectors from $\Lambda$, i.e., $d-1$ vectors from $\Z^d$ in our coordinates on $V$. This is equivalent to saying there exists a vector $a\in \Z^d$ such that $a_1b^1 + \ldots + a_db^d \in V$ is orthogonal to $H$ (one can find a rational orthogonal vector in our coordinate system and then scale it to be integer). 

Therefore, $\Lambda$-hyperplanes  in $V$ containing the origin are exactly the hyperplanes orthogonal to vectors in the dual lattice $\Lambda^*$ in $V$. Now one can see the equivalence of Theorem VII.8.3 on page 317 of~\cite{barvinok} and the statement of our theorem.
\end{proof}

We will also need a theorem about the structure of convex sets that contain no lattice points in their interior, originally stated in~\cite{Lovasz89}.

\begin{definition} A convex set $S\subseteq \R^n$ is said to be {\em lattice-free} if $\intt(S)\cap \Z^n = \emptyset$. A {\em maximal lattice-free set} is a lattice-free set that is not properly contained in another lattice-free set.
\end{definition}

\begin{theorem}\label{thm:mlfc-structure}\cite[Theorem 1.2]{bccz}[See also \cite{Lovasz89}]
%Let $V$ be a rational linear subspace of $\R^n$. 
A set $S\subset \R^n$ is a maximal lattice-free convex set  if and only if one of the following holds:
\begin{itemize}
\item[(i)] $S$ is a polyhedron of the form $S= P+L$ where $P$ is a polytope, $L$ is a rational linear space, $\dim(S)=\dim(P)+\dim(L)=n$, $S$ does not contain any integral point in its interior and there is an integral point in the relative interior of each facet of $S$;
\item[(ii)] $S$ is an irrational affine hyperplane of $\R^n$.
\end{itemize}

\end{theorem}

We now establish our main tool.

\begin{theorem} Let $A\in \R^{m \times n}$ be a rational matrix. Then for any $b \in \R^m$ such that $P:= \{x \in \R^n : Ax \geq b\}$ satisfies $P \cap \Z^n = \emptyset$, we must have $\ell(P,\Z^n, \R^n) \leq n^{5/2}$. [Note that $P$ is not assumed to be bounded]
\end{theorem}

\begin{proof} If $P$ is not full-dimensional, then $\aff P$ is given a system $\{ x : \tilde A x = \tilde b \}$ where the matrix $\tilde A$ is rational, since the matrix $A$ is rational and $\tilde A$ can be taken to be a submatrix of $A$. Now take a $\Z^n$-hyperplane $H$ that contains $\{ x | \tilde A x = 0 \}$. Then $\ell(P, \Z^n, \R^n, H) = 0$ or $1$, depending on whether the translate in which $P$ is contained in a $\Z^n$-hyperplane translate of $H$ or not. This immediately implies that $\ell(P,\Z^n,\R^n)$ is either $0, 1$. 

Thus, we focus on the case when $P$ is full-dimensional. By Theorem~\ref{thm:mlfc-structure}, there exists a basis $v^1, \ldots, v^n$ of $\Z^n$, a natural number $k \leq n$, and a polytope $C$ contained in the span of $v^1, \ldots, v^k$, such that $P \subseteq C + L$, where $L = span(\{v^{k+1}, \ldots, v^n\})$ and $(C + L) \cap \Z^n = \emptyset$ (the possibility of $k = n$ is allowed, in which case $L = \{0\}$). 

Define $V = span(\{v^1, \ldots, v^k\})$ and $\Lambda$ as the lattice formed by the basis $\{v^1, \ldots, v^k\}$. Since $C$ is a compact, convex set in $V$ and $C\cap \Lambda = \emptyset$, by Theorem~\ref{thm:flatness}, we must have that $\ell(C, \Lambda, V) \leq k^{5/2}$. Every $\Lambda$-hyperplane $H \subseteq V$ can be extended to a $\Z^n$-hyperplane $H' = H + L$ in $\R^n$. This shows that $\ell(C + L, \Z^n, \R^n) \leq k^{5/2} \leq n^{5/2}$. Since $P \subseteq C + L$, this gives the desired relation that $\ell(P,\Z^n, \R^n) \leq n^{5/2}$. \end{proof}

\begin{example}\label{ex:infinite-lattice-width}
If $A$ is not rational, the above result is not true. Consider the set
\begin{align*}
P := \{(x_1, x_2) : x_2 = \sqrt{2}(x_1-1/2)\}
\end{align*}
Now, $P \cap Z^2 = \emptyset$. Any $\Z^2$-hyperplane containing $(0,0)$ is the span of some integer vector. All such hyperplanes intersect $P$ in exactly one point, since the hyperplane defining $P$ has an irrational slope and so intersects every $\Z^2$-hyperplane in exactly one point. Hence, $\ell(P, \Z^2, \R^2) = \infty$ for all $H \in \mathcal H_\lambda$ and so $\ell(P, \Z^2, \R^2) = \infty$. 
\end{example}

This will help to establish bounds on the Chv\'atal rank of any lattice-free polyhedron with a rational constraint matrix. First we make the following modification of equation (6) on page 341 in~\cite{schrijver86}.

\begin{lemma}\label{lem:face-chvatal} Let $A\in \R^{m \times n}$ be a rational matrix. Let $b \in \R^m$ (not necessarily rational) and let $P:= \{x \in \R^n : Ax \geq b\}$. Let $F \subseteq P$ be a face. Then $F^{(t)} = P^{(t)} \cap F$ for any $t \in \N$.
\end{lemma}

\begin{proof} The proof follows the proof of (6) in~\cite{schrijver86} on pages 340-341 very closely. As observed in~\cite{schrijver86}, it suffices to show that $F' = P' \cap F$. 

Without loss of generality, we may assume that the system $Ax \geq b$ is TDI (if not, then throw in valid inequalities to make the description TDI). Let $F = P \cap \{x : \alpha x = \beta\}$ for some integral $\alpha \in \R^n$. 
The system $Ax \geq b, \alpha x \geq \beta$ is also TDI, which by Theorem 22.2 in~\cite{schrijver86} implies that the system $Ax \leq b, \alpha x = \beta$ is also TDI (one verifies that the proof of Theorem 22.2 does not need rationality for the right hand side). 

Now if $\beta$ is an integer, then we proceed as in the proof of the Lemma at the bottom of page 340 in~\cite{schrijver86}. 

If $\beta$ is not an integer, then $\alpha x \geq \lceil \beta \rceil$ and $\alpha x \leq \lfloor \beta \rfloor$ are both valid for $F'$, showing that $F' = \emptyset$. By the same token, $\alpha x \geq \lceil \beta \rceil$ is valid for $P'$. But then $P' \cap F = \emptyset$ because $\lceil \beta \rceil > \beta$. \end{proof}

We now prove the following modification of Theorem 23.3 from~\cite{schrijver86}.

\begin{theorem}\label{thm:schrijver-23.3} For every $n\in \N$, there exists a number $t(n)$ such that for any rational matrix $A\in \R^{m \times n}$ and $b \in \R^m$ (not necessarily rational) such that $P:= \{x \in \R^n : Ax \geq b\}$ satisfies $P \cap \Z^n = \emptyset$, we must have $P^{(t(n))} = \emptyset$. 
\end{theorem}

\begin{proof} We closely follow the proof in~\cite{schrijver86}. The proof is by induction on $n$. The base case of $n = 1$ is simple with $t(1) = 1$. Define $t(n) := n^{5/2} + 2 + (n^{5/2} + 1)t(n-1)$.

Since $P \cap \Z^n = \emptyset$, $\ell(P,\Z^n,\R^n) \leq n^{5/2}$ by Theorem~\ref{thm:flatness}, this means that there is an integral vector $c\in \R^n$ such that
 \begin{equation}\label{eq:width-bound}
 \lfloor \max_{x\in P} c^Tx\rfloor - \lceil\min_{x\in P} c^Tx \rceil \leq n^{5/2}.
 \end{equation}

Let $\delta = \lceil\min_{x\in P} c^Tx \rceil$. We claim that for each $k = 0, \ldots, n^{5/2} + 1$, we must have \begin{equation}\label{eq:slice} P^{(k+1 +k\cdot t(n-1))} \subseteq \{x : c^Tx \geq \delta + k\}.\end{equation}

For $k=0$, this follows from definition of $P'$. Suppose we know \eqref{eq:slice} holds for some $\bar k$; we want to establish it for $\bar k + 1$. So we assume $P^{(\bar k+1 +\bar k\cdot t(n-1))} \subseteq \{x : c^Tx \geq \delta + \bar k\}$. Now, since $P\cap \Z^n = \emptyset$, we also have $P^{(\bar k+1 +\bar k\cdot t(n-1))} \cap \Z^n = \emptyset$. Thus, the face $F = P^{(\bar k+1 +\bar k\cdot t(n-1))} \cap \{x : c^T x = \delta + \bar k\}$ satisfies the induction hypothesis and has dimension strictly less than $n$. By applying the induction hypothesis on $F$, we obtain that $F^{t(n-1)} = \emptyset$. By Lemma~\ref{lem:face-chvatal}, we obtain that $P^{(\bar k+1 +\bar k\cdot t(n-1) + t(n-1))} \cap \{x : c^T x = \delta + \bar k \} = \emptyset$. Thus, applying the Chvatal closure one more time, we would obtain that $P^{(\bar k+ 1 +\bar k\cdot t(n-1) + t(n-1) + 1)} \subseteq \{x : c^Tx \geq \delta + \bar k + 1) \}$. This confirms~\eqref{eq:slice} for $\bar k + 1$.

Using $k = n^{5/2} + 1$ in~\eqref{eq:slice}, we obtain that $P^{(n^{5/2}+2 +(n^{5/2} + 1)\cdot t(n-1))} \subseteq \{x : c^Tx \geq \delta + n^{5/2} + 1\}$. From~\eqref{eq:width-bound}, we know that $\max_{x\in P} c^Tx < \delta + n^{5/2} + 1$. This shows that $P^{(n^{5/2}+2 +(n^{5/2} + 1)\cdot t(n-1))} \subseteq P \subseteq \{x : c^Tx < \delta + n^{5/2} + 1\}$. This implies that $P^{(n^{5/2}+2 +(n^{5/2} + 1)\cdot t(n-1))} = \emptyset$, as desired.\end{proof}

We can now establish the following key result:

\begin{theorem}\label{theorem:schrijver-23.4}[See Schrijver ~\cite{schrijver86} Theorem 23.4]  For each rational matrix $A$ there exists a positive integer $t$ such that for every right-hand-side vector $b$ (not necessarily rational),
\begin{eqnarray}
\{ x \, : \, Ax \ge b \}^{(t)} = \{ x \, : \, Ax \ge b \}_{I}. \label{eq:t-bound}
\end{eqnarray}
\end{theorem}

\begin{proof} The proof proceeds exactly as the proof of Theorem 23.4 in~\cite{schrijver86}. The proof in~\cite{schrijver86} makes references to Theorems 17.2, 17.4 and 23.3 from~\cite{schrijver86}. Every instance of a reference to Theorem 23.3 should be replaced with a reference to Theorem~\ref{thm:schrijver-23.3} above. Theorems 17.2 and 17.4 do not need the rationality of the right hand side.
\end{proof}

We now have all the machinery we need to prove Theorem~\ref{theorem:b-j-mod}.

\begin{proof}[Proof of Theorem~\ref{theorem:b-j-mod}]
Given $A$ we can generate a nonnegative matrix $U$ using Theorem~\ref{theorem:basu} so that $UAx \ge Ub$ is TDI for all $b.$  Then by Theorem~\ref{theorem:schrijver-23.1} we get the \ch closure using the system  $UAx \ge \lceil Ub \rceil.$ Using Theorem~\ref{theorem:schrijver-23.4} we can apply this process $t$ times independent of $b$ and know we end up with $\{ x \, : \, Ax \ge b \}_{I}.$  We then apply Fourier-Motzkin elimination to this linear system and the desired $f_i$'s are obtained.
\end{proof}

\end{document}